\documentclass[reqno,centertags,12pt]{amsart}

\usepackage{amssymb,amsmath,amsthm,dsfont}

\newcommand{\R}{{\mathbb R}}
\newcommand{\C}{{\mathbb C}}
\newcommand{\Z}{{\mathbb Z}}
\newcommand{\N}{{\mathbb N}}
\newcommand{\Zz}{{\mathcal Z}}
\newcommand{\Tt}{{\mathfrak t}}
\newcommand{\F}{{\mathcal F}}
\newcommand{\Ff}{{\mathbf F}}
\newcommand{\ttau}{{\boldsymbol\tau}}
\newcommand{\Ll}{{\boldsymbol\lambda}}
\newcommand{\oo}{{\boldsymbol\omega}}
\newcommand{\OO}{{\boldsymbol\Omega}}

\theoremstyle{plain} \newtheorem{theorem}{Theorem}[section]
\newtheorem{lemma}[theorem]{Lemma}
\newtheorem{coro}[theorem]{Corollary}
\newtheorem{prop}[theorem]{Proposition}

\theoremstyle{definition} \newtheorem{definition}[theorem]{Definition}
\theoremstyle{remark} \newtheorem{remark}{Remark}

\newcommand{\secref}[1]{Section~\ref{#1}}
\newcommand{\thmref}[1]{Theorem~\ref{#1}}
\newcommand{\lemref}[1]{Lemma~\ref{#1}}
\newcommand{\corref}[1]{Corollary~\ref{#1}}
\newcommand{\propref}[1]{Proposition~\ref{#1}}
\newcommand{\defnref}[1]{Definition~\ref{#1}}

\newcommand{\la}{\langle}
\newcommand{\ra}{\rangle}
\newcommand{\ls}{\lesssim}

\newcommand{\XT}{\dot X^{s_p}_{T}}
\newcommand{\X}{\dot X^{s_p}}
\newcommand{\XTs}{\dot X^{s}_{T}}
\newcommand{\XTN}{\dot X^{0}_{T}}
\newcommand{\Ukdv}{U_{\textnormal{KdV}}}
\newcommand{\Vkdv}{V_{\textnormal{KdV}}}
\newcommand{\Bkdv}{B_{\textnormal{KdV}}}
\newcommand{\B}{{\dot B_\infty^{s_p,2}}}
\newcommand{\Bs}{\dot B^{s,2}_\infty}
\renewcommand{\l}{\lambda}
\newcommand{\p}{\partial}
\renewcommand{\d}[1]{d#1}
\newcommand{\dxdt}{dxdt}
\renewcommand{\L}[2]{\def\macroa{#1}\def\macrob{#2}\ifx\macroa\macrob{L_{t,x}^{#1}}\else{L_t^{#1}L_x^{#2}}\fi}

\begin{document}

\title[Well-posedness for supercritical gKdV]{Well-posedness
  for the supercritical gKdV equation}
\author[N.~Strunk]{Nils~Strunk}

\address{Universit\"at Bielefeld, Fakult\"at f\"ur Mathematik,
  Postfach 100131, 33501 Bielefeld, Germany}
\email{strunk@math.uni-bielefeld.de}

\begin{abstract}
  In this paper we consider the supercritical generalized
  Korteweg-de~Vries equation $\p_t\psi + \p_{xxx}\psi +
  \p_x(|\psi|^{p-1}\psi) = 0$, where $5\leq p\in\R$.  We prove a local
  well-posedness result in the homogeneous Besov space $\B(\R)$, where
  $s_p=\frac12-\frac{2}{p-1}$ is the scaling critical index. In
  particular local well-posedness in the smaller inhomogeneous Sobolev space
  $H^{s_p}(\R)$ can be proved similarly. As a byproduct a global well-posedness
  result for small initial data is also obtained.
\end{abstract}
\keywords{}
\maketitle

\section{Introduction}\label{sect:intro}
\noindent
  Consider the initial value problem associated to the generalized
  Korteweg-de~Vries (gKdV) equation, that is
  \begin{equation}
  \label{eq:i_gkdv}
    \left\{
      \begin{array}{rcl}
        \p_t\psi + \p_{xxx}\psi + \p_x\bigl(|\psi|^{p-1}\psi\bigr) & = & 0, \\
        \psi(0,x) & = & \psi_0(x).
      \end{array}
    \right.
  \end{equation}
  Well-posedness results of the Cauchy problem \eqref{eq:i_gkdv} (with
  $p\geq2$) has been studied by many authors in recent years.
  We want to give a brief overview of the best known well-posedness
  results.

  The fundamental work on this topic was done by Kenig, Ponce and Vega
  \cite{KPV93,KPV96} in 1993 and 1996. They proved local and small
  data global well-posedness for the sub-critical cases
  $p\in\{2,3,4\}$ in $H^s(\R)$ for certain $s$. For the KdV equation
  ($p=2$) they proved well-posedness for $s>-\frac34$. In the limiting
  case $s=-\frac34$ existence of solutions has been obtained by
  Christ, Colliander, and Tao \cite{CCT03}. Kenig, Ponce and Vega also
  proved well-posedness of the mKdV equation ($p=3$) for
  $s\geq\frac14$, and of the quartic gKdV equation ($p=4$) for
  $s\geq\frac{1}{12}$. So far the scaling space $H^{s_p}$ with
  $s_p=\frac12-\frac{2}{p-1}$ was not reached for the sub-critical
  cases. That changed in 2007, when Tao \cite{TAO07} proved local
  well-posedness (and global well-posedness for small data) of the
  quartic KdV equation in the scaling critical inhomogeneous Sobolev
  space $\dot H^{-\frac16}$. In 2012 Koch and Marzuola \cite{KM12}
  simplified and strengthened Tao's well-posedness result in the Besov
  space $\dot B^{-\frac16,2}_\infty$. For the supercritical cases
  $p\geq 5$, $p\in\N$, local well-posedness and global well-posedness
  for small data in the scaling critical spaces $\dot H^{s_p}$ was
  obtained by Kenig, Ponce and Vega in 1993. Recently Farah, Linares
  and Pastor extended the global well-posedness result for $p\geq 5$,
  $p\in\N$. In 2003 Molinet and Ribaud \cite{MR03} extended the
  well-posedness result in the supercritical cases to the homogeneous
  Besov space $\dot B^{s_p,2}_\infty(\R)$ with integer $p$. To our
  knowledge well-posedness results for non-integer $p\geq 5$ were not
  obtained so far. We present a unified proof of well-posedness in the
  homogeneous Besov space $\B(\R)$ for all $5\leq p\in\R$.

  In this paper we pick up techniques of Koch and Marzuola
  \cite{KM12} to prove local (and small data global) well-posedness for
  the supercritical gKdV equation, i.e.\ \eqref{eq:i_gkdv} with $5\leq
  p\in\R$.  The well-posedness is proved in the in the homogeneous Besov
  space $\B(\R)$ (see \defnref{def:besov}), where
\[
   s_p=\frac12-\frac{2}{p-1}
\]
  is the scaling critical exponent. The homogeneous Besov space $\B(\R)$
  is slightly larger than the scaling invariant homogeneous Sobolev
  space $\dot H^{s_p}(\R)$ consisting of all functions $u$ such that
\[
   \|u(t)\|_{\dot H^{s_p}} = \left( \sum_{\l\in1.01^{\Z}}\l^{2s_p}
     \|u_\l(t)\|_{L^2}^2 \right)^{1/2} < \infty.
\]
  Here, $u_\l$ denotes the Littlewood-Paley decomposition of $u$ at frequency $\l$ 
  that is defined in \secref{sect:funct_spaces}.

  In the following, let $v$ be a solution to the Airy equation with same
  initial data
  \begin{equation}
  \label{eq:i_airy}
    \left\{
      \begin{array}{rcl}
        \p_tv + \p_{xxx}v & = & 0, \\
        v(0,x) & = & \psi_0(x).
      \end{array}
    \right.
  \end{equation}
  For the quartic gKdV equation
  \begin{equation}
  \label{eq:i_km_gkdv}
    \left\{
      \begin{array}{rcl}
        \p_t\psi + \p_{xxx}\psi + \p_x(\psi^4) & = & 0, \\
        \psi(0,x) & = & \psi_0(x),
      \end{array}
    \right.
  \end{equation}
  Koch and Marzuola \cite{KM12} proved the following local well-posedness result:
  \begin{theorem}[Koch and Marzuola \cite{KM12}]
  \label{thm:km_lwp}
    Let $r_0>0$. Then there exist $\varepsilon_0,\delta_0>0$ such that, if
    $0<T\leq\infty$,
\[
    \|\psi_0\|_{\dot B_\infty^{-\frac16,2}}\leq r_0
\]
    and
\[
    \sup_{\l\in1.01^{\Z}} \|v_\l\|_{L^6([0,T],\R)}\leq\delta_0
\]
    then there is an unique solution $\psi=v+w$ to \eqref{eq:i_km_gkdv} with
\[
    \|w\|_{\dot X_{\infty,T}^{-\frac16}}\leq\varepsilon_0.
\]
    Moreover, the function $w$ (and hence $\psi$) depends analytically on 
    the initial data.
  \end{theorem}
  From this local well-posedness result they even obtained global well-posedness
  for small data $\psi_0$, since one easily proves by Strichartz' estimates
  and the definition of the spaces that 
\[
   \sup_{\l\in1.01^{\Z}} \|v_\l\|_{L^6([0,T],\R)}\ls\|\psi_0\|_{\dot B_\infty^{-\frac16,2}}.
\] 

  In the sequel, we are going to prove the analogue statement in the
  supercritical case, i.e.\ for \eqref{eq:i_gkdv} with $5\leq p\in\R$:
  \begin{theorem}
  \label{thm:lwp}
  Let $5\leq p\in\R$, $s_p=\frac12-\frac{2}{p-1}$ and $r_0>0$. Then there exist
  $\varepsilon_0,\delta_0>0$ such that, if $0<T\leq\infty$,
\[
   \|\psi_0\|_{\B} \leq r_0
\]
     and
\begin{equation}
\label{eqn:smallness_cond}
   \sup_{\l\in1.01^{\Z}} \l^{\frac16+s_p} \|v_\l\|_{L^6([0,T],\R)} \leq \delta_0,
\end{equation}
    then there exists an unique solution $\psi=v+w$ to \eqref{eq:i_gkdv} with
\[
    \|w\|_{\XT} \leq \varepsilon_0.
\]
    Moreover, the solution map is Lipschitz continuous.
  \end{theorem}

  Using the same arguments as Koch and Marzuola, we obtain global
  well-posedness for small initial data $\psi_0$ as well:
\begin{coro}
\label{cor:gwp}
  Let $5\leq p\in\R$, $s_p=\frac12-\frac{2}{p-1}$ and $\delta_0(1)$
  be the $\delta_0$ of
  \thmref{thm:lwp}, which depends on $r_0$,
  evaluated at $r_0=1$. Let $\kappa_0$ and $\kappa_1$ be the constants
  from \lemref{lem:strichartz} and
  \lemref{lem:airy_est_XT}, respectively. Then there exists
  $\varepsilon_0 > 0$ such that for
\[
    \|\psi_0\|_{\B}\leq \min\left\{1,\frac{\delta_0(1)}{\kappa_0\kappa_1}\right\}
\]
  there is an unique solution $\psi=v+w$ to \eqref{eq:i_gkdv} with
\[
    \|w\|_{\X}\leq\varepsilon_0.
\]
  Moreover, the solution map is Lipschitz continuous.
\end{coro}

  The main ingredient of the proof of \thmref{thm:lwp} is a
  multi-linear estimate that gives bounds on the Duhamel term of the
  nonlinearity. A crucial tool to get these estimates are the recently
  introduced $U^p$ and $V^p$ spaces.  The rest of the proof is a
  standard fixed point argument to get existence and
  uniqueness. However, due to the non-integer exponents, this
  argument gets a bit more delicate.
  \begin{remark}
  \label{rem:sobolev_spaces}
  The analogue local and global well-posedness in the inhomogeneous
  Sobolev space $H^{s_p}$ follows along these lines. Note that the
  function spaces and the summation has to be modified.
  \end{remark}
  \begin{remark}
  \label{rem:small_cond}
  It is possible to choose different H\"older exponents in the proof
  of the multi-linear estimates (\lemref{lem:multlin_est} and
  \lemref{lem:multlin_est2}) and hence to require an other smallness
  condition replacing the smallness condition \eqref{eqn:smallness_cond}
  of the linear solution.
  \end{remark}

  Throughout this paper, we will use mixed Lebesgue spaces $\L{p}{q}$ which are
  defined via the norm
\[
   \|f\|_{\L{p}{q}}=\left(\int \|f(t,\cdot)\|_{L_x^q}^p\d{t}\right)^{\frac1p},
   \quad 1\leq p < \infty,\;1\leq q\leq\infty,
\]
  and with obvious modifications for $p=\infty$. If $p=q$, then we
  write $\L{p}{p}$ for brevity. Moreover, we want to mention that
  we write $A\ls B$, if there is a harmless constant $c>0$ such that
  $A\leq cB$. 

  This paper is organized as follows: In \secref{sect:funct_spaces} we
  give a brief introduction to the function spaces used in this paper.
  \secref{sec:wp_estimates_u_v} provides some basic linear and bilinear
  estimates. Multi-linear estimates
  to control the Duhamel term of the nonlinearity
  are proved in \secref{sect:multilin_est}. \thmref{thm:lwp} and the
  global well-posedness result is proved in \secref{sect:proof}.

{\sc Acknowledgments} 
This paper is an extension of the diploma thesis of the author.
The author wishes to thank the thesis advisor Herbert Koch and
Sebastian Herr for helpful comments while working on this result.

\section{Function spaces}\label{sect:funct_spaces}
\noindent
Crucial tools to prove this well-posedness results are the
function spaces $U^p$, which have been introduced in the context of dispersive
PDEs by Tataru and Koch-Tataru \cite{KT05,KT07} as well as the closely related
spaces of bounded $p$-Variation $V^p$ due to Wiener \cite{WIE24}.
The following exposition of the $U^p$ and $V^p$ spaces may be found
in \cite{HHK09}. We refer the reader to this paper for detailed
definitions and proofs.

We consider functions taking values in $L^2=L^2(\R^d,\R)$, but in the
general part of this section one may replace $L^2$ by an arbitrary Hilbert
space. Let $\Zz$ be the set of finite partitions
$-\infty<t_0<t_1<\ldots<t_K\leq\infty$.

\begin{definition}
\label{def:u}
  Let $1\leq p<\infty$. For $\{t_k\}_{k=0}^K\in\Zz$ and
  $\{\phi_k\}_{k=0}^{K-1}\subset L^2$ with
  $\sum_{k=0}^{K-1}\|\phi_k\|_{L^2}^p=1$ and $\phi_0=0$, we call the function
  $a:\R\to L^2$ given by
  \begin{equation*}
  \label{eqn:u_atom}
    a=\sum_{k=1}^K\chi_{[t_{k-1},t_k)}\phi_{k-1}
  \end{equation*}
  a $U^p$-atom. Furthermore, we define the atomic space
  \begin{equation*}
  \label{eqn:atomic_space}
    U^p=\left\{u=\sum_{j=1}^\infty \l_j a_j : a_j~U^p
     \text{-atom, }\l_j\in\C,
     \text{ s.\,t.\ }\sum_{j=1}^\infty\left|\l_j\right|<\infty\right\}
  \end{equation*}
  endowed with the norm
  \begin{equation*}
  \label{eqn:u-norm}
    \|u\|_{U^p}:=\inf\left\{\sum_{j=1}^\infty |\l_j| :
    u=\sum_{j=1}^\infty \l_j a_j\in\C,
    \text{ s.\,t.\ }\sum_{j=1}^\infty \left|\l_j\right|<\infty\right\}.
  \end{equation*}
\end{definition}

Two useful statements about $U^p$ are collected in the following
\begin{prop}
\label{prop:u}
  Let $1\leq p < q < \infty$.
  \begin{enumerate}
    \item\label{it:u_banach} $\|\cdot\|_{U^p}$ is a norm. The space $U^p$ is
      complete and hence a Banach space.
    \item\label{it:u_emb} The embeddings $U^p \subset U^q \subset
      L^\infty(\R,L^2)$ are continuous. 
  \end{enumerate}
\end{prop}

\begin{definition}
\label{def:v}
  Let $1\leq p<\infty$.
  \begin{enumerate}
  \item We define $V^p$ as the normed space of all functions $v:\R \to
    L^2$ such that $\lim_{t\to\pm\infty}v(t)$ exists and for which the
    norm
    \begin{equation*}
      \label{eqn:hom_norm_v}
      \|v\|_{V^p}:=\sup_{\{t_k\}_{k=0}^K \in \Zz}
      \left(\sum_{k=1}^{K}
        \|v(t_{k})-v(t_{k-1})\|_{L^2}^p\right)^{\frac{1}{p}}
    \end{equation*}
    is finite. We use the convention that
    $v(-\infty)=\lim_{t\to-\infty}v(t)$ and $v(\infty)=0.$
  \item We denote the closed subspace of all right-continuous functions
    $v:\R\to L^2$ such that $\lim_{t\to -\infty}v(t)=0$ by $V_{rc}^p$.
  \end{enumerate}
\end{definition}

\begin{remark}
\label{rem:v}
  Note that we set $v(\infty)=0$, which may differ from the limit of $v$ at
  $\infty$.
\end{remark}

\begin{prop}
\label{prop:v}
  Let $1\leq p<q <\infty$.
  \begin{enumerate}
    \item\label{it:v_emb1} The embedding $U^p\subset V_{rc}^p$ is continuous.
    \item\label{it:v_emb2} The embeddings $V^p\subset V^q$ are continuous.
    \item\label{it:v_emb3} The embedding $V_{rc}^p\subset U^q$ is continuous,
       and
       \begin{equation*}
       \label{eqn:u_v_embed_est}
         \|v\|_{U^q} \leq c_{p,q}\|v\|_{V^p}.
       \end{equation*}
  \end{enumerate}
\end{prop}

\begin{prop}
\label{prop:bilin_from} 
  For $u\in U^p$ and $v\in V^{p'}$, where $1=\frac1p+\frac{1}{p'}$, and a
  partition $\Tt:=\{t_k\}_{t=0}^K\in\Zz$ we define
  \begin{equation*}
  \label{eqn:pre_bilin_from}
    B_\Tt(u,v) := \sum_{k=1}^K \bigl\la u(t_{k-1}), v(t_k)-v(t_{k-1}) \bigr\ra,
  \end{equation*}
  where $\la \cdot,\cdot \ra$ denotes the inner product of $L^2$. Notice that
  $v(t_K)=0$ since $t_K=0$ for all partitions $\{t_k\}_{t=0}^K\in\Zz$. There is
  a unique number $B(u,v)$ with the property that for all $\varepsilon > 0$
  there exists $\Tt\in\Zz$ such that for every $\Tt'\subset\Tt$ it holds
  \begin{equation*}
  \label{eqn:bilin_approx}
    \left| B_{\Tt'}(u,v) - B(u,v) \right| < \varepsilon,
  \end{equation*}
  and the associated bilinear form
  \begin{equation*}
  \label{eqn:bilin_map}
    B: U^p \times V^{p'}: (u,v)\mapsto B(u,v)
  \end{equation*}
  satisfies the estimate
  \begin{equation*}
  \label{eqn:bilin_est}
    \left| B(u,v) \right| \leq \|u\|_{U^p}\|v\|_{V^{p'}}.
  \end{equation*}
\end{prop}

\begin{prop}
\label{prop:v_u_duality}
  Let $1 < p < \infty$. We have
  \begin{equation*}
  \label{eqn:v_u_duality}
    \left(U^p\right)^* = V^{p'}
  \end{equation*}
  in the sense that
  \begin{equation*}
  \label{eqn:duality_map}
    T:V^{p'}\to\left(U^p\right)^*,\quad T(v):=B(\cdot,v)
  \end{equation*}
  is an isometric isomorphism.
\end{prop}

\begin{coro}
\label{cor:duality_u_v}
  For $1<p<\infty$, $u\in U^p$ and for $v\in V^p$ the following estimates hold
  true
  \begin{align*}
    \|u\|_{U^p} &= \sup_{\substack{v\in V^{p'}\\ \|v\|_{V^{p'}}=1}} |B(u,v)|
  \intertext{and}
    \|v\|_{V^p} &= \sup_{u\textnormal{ }U^{p'}\textnormal{-atom}} |B(u,v)|.
  \end{align*}
\end{coro}

\begin{prop}
\label{prop:bilin_form_int}
  Let $1 < p < \infty$. If the distributional derivative of $u$ is in $L^1$
  and $v\in V^{p}$. Then,
  \begin{equation*}
  \label{eqn:bilin_form_int}
    B(u,v) = -\int_{-\infty}^{\infty} \bigl\la u'(t), v(t) \bigr\ra \d{t}.
  \end{equation*}
\end{prop}

Following Bourgain's strategy for the Fourier restriction spaces we adapt the
$U^p$ and $V^p$ space to the gKdV equation.
\begin{definition}
\label{def:u_v_kdv}
  Define the Airy group $S: C(\R,L^2)\to C(\R,L^2)$ as
  \begin{equation*}
  \label{eqn:airy_group}
    S(t) := e^{-t\p_x^3} = \F_x^{-1}e^{-it\xi^3}\F_x,
  \end{equation*}
  where $\F_x$ denotes the Fourier transform with respect to $x$.
  For $u\in C(\R,L^2)$ we set $v(t):=S(-t)u(t)$ and define
  \begin{equation*}
    \Ukdv^p := SU^p\quad\text{and}\quad\Vkdv^p := SV_{rc}^p,
  \end{equation*}
  with norms
  \begin{equation*}
  \label{eqn:u_v_kdv_norms}
    \|u\|_{\Ukdv^p} = \|v\|_{U^p}
    \quad\text{and}\quad
    \|u\|_{\Vkdv^p} = \|v\|_{V^p}.
  \end{equation*}
  Again, we define a
  bilinear map $\Bkdv$ such that for $u\in \Ukdv^p$, $v\in \Vkdv^{p'}$,
  we have for a function $u$ with $(\p_t+\p_x^3)u\in L^1L^2$
  \begin{equation*}
  \label{eqn:bilin_map_kdv}
    \Bkdv(u,v)=-\int \bigl\la (\p_t+\p_x^3)u,v \bigr\ra \d{t}.
  \end{equation*}
  By this bilinear map, we obtain similar duality statements as in
  \corref{cor:duality_u_v}.
\end{definition}

For minor technical purposes we use a slight unusual Littlewood-Paley
decomposition, using powers of $1.01$ instead of $2$ (cf.\ \cite{KM12,TAO07}): 
We fix a nonnegative, even function $\phi\in C_0^\infty((-2,2))$ with
$\phi(s)=1$ for $|s|\leq1$. We use this function to define a partition
of unity: for $\l\in1.01^{\Z}$, we set
\[
   \Psi_\l(\xi) = \phi\left(\frac{|\xi|}{\l}\right)
     - \phi\left(\frac{1.01|\xi|}{\l}\right).
\]
We define the Littlewood-Paley operators $P_\l:L^2(\R)\to L^2(\R)$ as the
Fourier multiplier with symbol $\Psi_\l$. For brevity we write
$u_\l:=P_\l u$. Furthermore, we define
\[
   u_{\leq\l}:=P_{\leq\l}u:=\sum_{\substack{\mu\leq\l\\\mu\in1.01^{\Z}}} P_{\mu}u
   \quad\text{and}\quad
   u_{<\l}:=P_{<\l}u:=(P_{\leq\l}-P_{\l})u.
\]
\begin{definition}
\label{def:besov}
  For $s\in\R$, we define the homogeneous Besov spaces $\dot
  B_{\infty}^{s,2}$ as the set of all tempered distributions on $\R^n$
  for which the norm
\[
   \|v\|_{\dot B_{\infty}^{s,2}}=\sup_{\l\in1.01^\Z}\l^s\|v_\l\|_{L^2_x}
\]
  is finite.
\end{definition}
We pic up the homogeneous space $\dot X^s$ that was defined in
\cite{KM12}.
\begin{definition}
\label{def:xt}
  For $s\in\R$, we define the real-valued homogeneous space $\dot X^s$ using the
  norm
\[
  \|v\|_{\dot X^s} = \sup_{\l\in1.01^{\Z}} \l^s \|v_\l\|_{\Vkdv^2}.
\]
  Furthermore, we denote by $\dot X_{T}^s$ the functions on the
  time space set $(0,T)\times\R$.
\end{definition}

The following estimate follows directly from the definition of the
spaces $\Ukdv^p$ and $\Vkdv^p$.

\begin{lemma}
\label{lem:airy_est_XT}
  Let $v$ be a solution to the Airy equation
  \begin{equation*}
  \label{eqn:airy_est_XT_eqn}
  \left\{
    \begin{array}{rcl}
      \p_t v+\p_{xxx}v & = & f, \\
      v(0,x) & = & v_0(x),
    \end{array}
  \right.
  \end{equation*}
  then, for $s\in\R$, there exists $\kappa_1>0$ such that the
  following estimate holds true
  \begin{equation*}
  \label{eqn:airy_est_XT_est}
    \|v\|_{\XTs} \leq \kappa_1 \left( \|v_0\|_{\Bs} + \sup_{\l\in1.01^{\Z}} \l^{s}
      \left\| \int_{0}^t S(t-s) f_\l(s) \d{s} \right\|_{\Ukdv^2} \right).
  \end{equation*}
\end{lemma}

\section{Linear and bilinear estimates}
\label{sec:wp_estimates_u_v}
\noindent
The following Lemma is based on \cite{KPV93} and may be found in
\cite[formula (3.2) and (7.7)]{KM12}.
\begin{lemma}[Strichartz' estimates]
\label{lem:strichartz}
  Let $u\in \Ukdv^q$, $q>4$ and $(q,r)$ be a Strichartz pair of the
  Airy equation, i.e.\ $\frac2q+\frac1r=\frac12$. Then,
  \begin{equation}
  \label{eqn:strichartz_est_up}
     \|u\|_{\L{q}{r}} \ls \bigl\| |D_x|^{-\frac1q}u \bigr\|_{\Ukdv^q}.
  \end{equation}
  In particular, for $\l\in1.01^{\Z}$ we have
  \begin{align*}
  \label{eqn:strichartz_est_6}
    \|u_\l\|_{\L{q}{r}} &\ls \l^{-\frac1q}\|u_\l\|_{\Ukdv^q}
      \ls \l^{-\frac1q}\|u_\l\|_{\Vkdv^2}.
  \end{align*}
  Hence for $s\in\R$ it holds that
  \begin{equation*}
  \label{eqn:strichartz_est_6_1}
    \sup_{\l\in1.01^{\Z}} \l^{\frac1q+s} \|u_\l\|_{\L{q}{r}} \leq \kappa_0 \|u\|_{\XTs}.
  \end{equation*}
\end{lemma}

\begin{lemma}[{\cite[page 179]{KM12}}]
\label{lem:linfty_est}
  Let $u\in\XT$, $\l\in1.01^\Z$, then we have for all $p\geq5$
\[
   \|u_{\leq\l}\|_{\L{\infty}{\infty}}
   \ls\l^{\frac12-s_p}\|u\|_{\XT}.
\]
\end{lemma}
\begin{proof}
  This estimate follows directly from Bernstein's inequality and
  the energy estimate.
\end{proof}

The next Corollary immediately follows from interpolating the
$\L{6}{6}$-Strichartz estimate and the $\L{\infty}{\infty}$ estimate. 
\begin{coro}
\label{cor:linear_lp}
  Let $u\in\Vkdv^2$, $\l\in1.01^\Z$ and $q\geq6$, then we have
\[
   \|u_\l\|_{\L{q}{q}}\ls\l^{\frac12-\frac4q}\|u_\l\|_{\Vkdv^2},
\]
  and if $p\geq5$, then we even have for all $q>2(p-1)$
\[
   \|u_{\leq\l}\|_{\L{q}{q}}\ls\l^{\frac12-\frac4q-s_p}\|u_{\leq\l}\|_{\XT}.
\]
\end{coro}

The following bilinear estimate is based on a bilinear estimate of
Gr\"unrock \cite{GRU05} and can be found in \cite[formula (7.8)]{KM12}.
\begin{lemma}[Bilinear estimate]
\label{lem:bilinear}
  Let $u,v\in \Ukdv^2$ and let $\l,\mu\in1.01^\Z$ such that $\l\geq1.1\mu$. Then
  \begin{equation*}
  \label{eqn:bilinear_est}
    \|v_\mu u_\l\|_{\L{2}{2}}
    \lesssim \l^{-1}\|v_\mu\|_{\Ukdv^2}\|u_\l\|_{\Ukdv^2}.
  \end{equation*}
\end{lemma}

\begin{coro}
  Let $2< q\leq \infty$ and $\l\geq1.1\mu$. Then for $u,v\in\XT$
\[
   \|v_\mu u_\l\|_{\L{q}{q}}
   \ls \mu^{\frac12-\frac1q-s_p}\l^{\frac12-\frac3q-s_p} \|v\|_{\XT} \|u\|_{\XT}.
\]
  If in addition $p\geq5$, then for all $q>\frac{p-1}{2}$ we may even estimate
\begin{equation}
\label{eqn:bilin_est}
   \|v_{\leq\mu} u_\l\|_{\L{q}{q}}
   \ls \mu^{\frac12-\frac1q-s_p} \l^{\frac12-\frac3q-s_p} \|v\|_{\XT}\|u\|_{\XT}.
\end{equation}
\end{coro}
\begin{proof}
  The first inequality follows by interpolating the bilinear estimate
  (\lemref{lem:bilinear}) and the $\L{\infty}{\infty}$ estimate
  (\lemref{lem:linfty_est}) as well as \propref{prop:v}.
  As a consequence the second inequality
  simply follows from a Littlewood-Paley decomposition of $v_{\leq\mu}$.
  Note that in \eqref{eqn:bilin_est} $q$ is choosen such that the
  exponent of $\mu$ is larger than zero.
\end{proof}
  
\section{Multi-linear estimates}\label{sect:multilin_est}
\noindent
\begin{lemma}
\label{lem:multlin_est}
  Let $5\leq p\in\R$ and $\l_2\leq\ldots\leq\l_5\in1.01^\Z$, $\mu\in1.01^\Z$
  and $1.1\l_5>\mu$. There exists $r>0$ independent of $T$ such that for given
  $v_i,u\in\XT$, $i=0,\ldots,5$, we have for some small
  $\varepsilon>\delta>0$
\begin{multline*}
   \left|
     \int |v_{0,\leq\l_2}|^{p-5}v_{1,\leq\l_2}v_{2,\l_2}\cdots v_{5,\l_5}u_{\mu} \dxdt
   \right| \\
   \leq r \l_2^{\delta}\l_5^{-\varepsilon-s_p} \mu^{-1-\delta+\varepsilon}
       \|v_0\|_{\XT}^{p-5} \prod_{i=1}^5 \|v_i\|_{\XT} \|u_\mu\|_{\Vkdv^2}.
\end{multline*}
Moreover, we may even replace one factor $\|v_i\|_{\XT}$, $i=3,4$, on the
right hand side by
\[
   \sup_{\l\in1.01^{\Z}}\l^{\frac16+s_p}\|v_{i,\l}\|_{L^6([0,T],\R)}.
\]
\end{lemma}
\begin{proof}
  In order to prove this multi-linear estimate, we distinguish two
  cases. First, we consider the case when all frequencies $\l_i$ are 
  comparable, i.e.\ $\l_5\leq 1.1\l_2$.
  In the second case, we consider the situation if the frequency $\l_5$
  is much greater than $\l_2$, i.e.\ $1.1\l_2 < \l_5$. In this situation,
  we can make use of the strong bilinear estimate.

  \emph{1st case:} $1.1\l_2 \geq \l_5$ \\
  We start by assuming that $p>5$ and consider the integral
\[
   \left|
    \int |v_{0,\leq\l_2}|^{p-5}v_{1,\leq\l_2}v_{2,\l_2}\cdots v_{5,\l_5}u_{\mu} \dxdt
   \right|.
\]
  Let $q=2(p-4)$, then using H\"older's inequality we can estimate
  the integral by
\[
   \|v_{0,\leq\l_2}\|_{\L{\infty}{q}}^{p-5}
   \|v_{1,\leq\l_2}\|_{\L{\infty}{q}}
   \|v_{2,\l_2}\|_{\L{9/2}{18}}
   \|v_{3,\l_3}\|_{\L{6}{6}}\|v_{4,\l_4}\|_{\L{6}{6}}
   \|v_{5,\l_5}\|_{\L{9/2}{18}}\|u_{\mu}\|_{\L{9/2}{18}}.
\]
  By Sobolev embeddings, the energy estimate and the definition of
  $\XT$, we obtain
\begin{align*}
   \|v_{i,\leq\l_2}\|_{\L{\infty}{q}}
   \ls \l_2^{\frac12-\frac1q-s_p}\|v_i\|_{\XT},
   \quad i=0,1.
\end{align*}
  Strichartz' estimates allows to determine
\begin{align*}
   \|v_{i,\l_i}\|_{\L{9/2}{18}}&\ls \l_i^{-\frac29-s_p} \|v_i\|_{\XT},
   \quad i=2,5,\\
   \|v_{i,\l_i}\|_{\L{6}{6}}&\ls \l_i^{-\frac16-s_p} \|v_i\|_{\XT},
   \quad i=3,4,
   \intertext{as well as}
   \|u_{\mu}\|_{\L{9/2}{18}}&\ls \mu^{-\frac29} \|u_{\mu}\|_{\Vkdv^2}.
\end{align*}
  Since $\l_2$ and $\l_5$ are comparable, the product of the terms of
  $\l_i$ can be estimated by a constant times
  $\l_2^{\frac29}\l_5^{-1-s_p}$ for instance.

  If $p=5$, then we split the integral into two terms, namely
\begin{align*}
  I_1+I_2 =
  \left|
    \int v_{1,\ll\l_2}v_{2,\l_2}\cdots v_{5,\l_5}u_{\mu} \dxdt
   \right|
   +\left|
    \int v_{1,\sim\l_2}v_{2,\l_2}\cdots v_{5,\l_5}u_{\mu} \dxdt
   \right|,
\end{align*}
  where $v_{1,\ll\l_2}:=\sum_{\l_1:\;1.1\l_1<\l_2}v_{1,\l_1}$ and
  $v_{1,\sim\l_2}:=v_{1,\leq\l_2}-v_{1,\ll\l_2}$.
  Using  H\"older's inequality we may estimate $I_1$ and obtain
\[
   \|v_{1,\ll\l_2}v_{2,\l_2}\|_{\L{5/2}{5/2}}
   \|v_{3,\l_3}\|_{\L{10}{10}}
   \|v_{4,\l_4}\|_{\L{6}{6}}
   \|v_{5,\l_5}\|_{\L{6}{6}}
   \|u_\mu\|_{\L{6}{6}}.
\]
  Applying the bilinear estimate, \corref{cor:linear_lp} and Strichartz'
  estimates yields
\begin{align*}
  \|v_{1,\ll\l_2}v_{2,\l_2}\|_{\L{5/2}{5/2}}
  &\ls \l_2^{-\frac{3}{5}}\|v_1\|_{\XTN}\|v_2\|_{\XTN},\\
  \|v_{3,\l_3}\|_{\L{10}{10}}
  &\ls \l_3^{\frac{1}{10}}\|v_3\|_{\XTN},\\
  \|v_{i,\l_i}\|_{\L{6}{6}}
  &\ls \l_i^{-\frac16}\|v_i\|_{\XTN},\quad i=4,5,\\
  \|u_\mu\|_{\L{6}{6}}
  &\ls \mu^{-\frac16}\|u_\mu\|_{\Vkdv^2}.
\end{align*}
  $I_2$ simply can be estimated by
\[
   \|v_{1,\sim\l_2}\|_{\L{6}{6}}
   \|v_{2,\l_2}\|_{\L{6}{6}}
   \|v_{3,\l_3}\|_{\L{6}{6}}
   \|v_{4,\l_4}\|_{\L{6}{6}}
   \|v_{5,\l_5}\|_{\L{6}{6}}
   \|u_\mu\|_{\L{6}{6}}.
\]
  using H\"older's inequality, which can be further estimated by Strichartz'
  estimates.
  Since $1.1\l_2\geq\l_5$, the product of the $\l_i$ frequencies can be estimated
  by, e.g., $\l_2^{\frac16}\l_5^{-1}$.

  \emph{2nd case:} $1.1\l_2 < \l_5$ \\
  The main idea in this situation is to use the strong bilinear estimate, which
  allows to bound
\[ 
   \|v_{2,{\l_2}}v_{5,\l_5}\|_{\L{q}{q}}
   \ls\l_2^{\frac12-\frac1q}\l_5^{\frac12-\frac3q}
      \|v_{2}\|_{\XT}\|v_{5}\|_{\XT}
\]
  provided $2<q\leq\infty$.  We define the H\"older exponents
  $q_1=2(p+5)$ and $q_2=2+\frac{2}{p+4}$. Using H\"older's
  inequality, we may bound
\[
   \left|
     \int |v_{0,\leq\l_2}|^{p-5}v_{1,\leq\l_2}v_{2,\l_2}\cdots v_{5,\l_5}u_{\mu} \dxdt
   \right|
\]
  by
\[
   \|v_{0,\leq\l_2}\|_{\L{\infty}{\infty}}^{p-5}
   \|v_{1,\leq\l_2}\|_{\L{q_1}{q_1}}
   \|v_{3,\l_3}\|_{\L{6}{6}}
   \|v_{4,\l_4}\|_{\L{6}{6}}
   \|u_\mu\|_{\L{6}{6}}
   \|v_{2,\l_2}v_{5,\l_5}\|_{\L{q_2}{q_2}}.
\]
  From \lemref{lem:linfty_est} and the definition of $\XT$, we obtain
\[
   \|v_{0,\leq\l_2}\|_{\L{\infty}{\infty}}^{p-5}
   \ls \l_2^{(p-5)(\frac12-s_p)}\|v_0\|_{\XT}^{p-5}.
\]
  \corref{cor:linear_lp} allows to estimate
\[
   \|v_{1,\leq\l_2}\|_{\L{q_1}{q_1}}
   \ls \l_2^{\frac12-\frac{4}{q_1}-s_p}\|v_1\|_{\XT}.
\]
  By Strichartz' estimates, we obtain
\[
   \|v_{i,\l_i}\|_{\L{6}{6}}
   \ls \l_i^{-\frac16-s_p} \|v_i\|_{\XT}
\]
  for $i=3,4$, as well as
\[
   \|u_{\mu}\|_{\L{6}{6}} \ls \mu^{-\frac16} \|u_{\mu}\|_{\Vkdv^2}.
\]
  Furthermore, since $q_2>2$ the bilinear estimate gives
\[ 
   \|v_{2,\l_2}v_{5,\l_5}\|_{\L{q_2}{q_2}}
   \ls\l_2^{\frac12-\frac{1}{q_2}-s_p}\l_5^{\frac12-\frac{3}{q_2}-s_p}\|v_2\|_{\XT}\|v_5\|_{\XT}.
\]
  The product of the $\l_i$ can be estimated by
  $\l_2^{\frac{1}{6}-\frac{3}{2(p+5)}}\l_5^{-1+\frac{3}{2(p+5)}-s_p}$.
  Note that the exponent of $\l_2$ is bigger than zero for all $p\geq5$.
\end{proof}

If we assume that the frequency $\mu$ is much greater than all other
frequencies, then we can even prove the following Lemma.
\begin{lemma}
\label{lem:multlin_est2}
  Let $5< p\in\R$ and $\l_2\leq\ldots\leq\l_5\in1.01^\Z$, $\mu\in1.01^\Z$
  and $1.1\l_5\leq\mu$. There exists $r>0$ independent of $T$ such that for given
  $v_i,u\in\XT$, $i=0,\ldots,5$, we have
\begin{multline*}
   \left|
     \int |v_{0,\leq\l_2}|^{p-5}v_{1,\leq\l_2}v_{2,\leq\l_2}v_{3,\l_3}v_{4,\l_4}v_{5,\l_5}u_{\mu}
     \dxdt
   \right| \\
   \leq r \l_2^{\frac{1}{15}}\l_5^{-\frac16-s_p} \mu^{-\frac{9}{10}}
       \|v_0\|_{\XT}^{p-5} \prod_{i=1}^5 \|v_i\|_{\XT} \|u_\mu\|_{\Vkdv^2}.
\end{multline*}
  Moreover, we may even replace one factor $\|v_i\|_{\XT}$, $i=3,4$, on the
  right hand side by
\[
   \sup_{\l\in1.01^{\Z}}\l^{\frac16+s_p}\|v_{i,\l}\|_{L^6([0,T],\R)}.
\]
\end{lemma}
\begin{proof}
  The proof is quite similar to the proof of \lemref{lem:multlin_est}.
  We consider
\[
   \left|
     \int |v_{0,\leq\l_2}|^{p-5}v_{1,\leq\l_2}v_{2,\leq\l_2}v_{3,\l_3}v_{4,\l_4}v_{5,\l_5}u_{\mu}
     \dxdt
   \right|
\]
  and define the H\"older exponent $q=5(p-3)$.
  Using H\"older's inequality we estimate
\[
   \|v_{0,\leq\l_2}\|_{\L{q}{q}}^{p-5}
   \|v_{1,\leq\l_2}\|_{\L{q}{q}}
   \|v_{2,\leq\l_2}\|_{\L{q}{q}}
   \|v_{4,\l_4}\|_{\L{6}{6}}
   \|v_{5,\l_5}\|_{\L{6}{6}}
   \|v_{3,\l_3}u_\mu\|_{\L{15/7}{15/7}}.
\]
  By \corref{cor:linear_lp} and the definition of $\XT$, we obtain
\begin{align*}
   \|v_{i,\leq\l_2}\|_{\L{q}{q}}
   \ls \l_2^{\frac12-\frac4q-s_p}\|v_i\|_{\XT},
   \quad i=0,1,2.
\end{align*}
  Applying Strichartz' estimates yields
\[
   \|v_{i,\l_i}\|_{\L{6}{6}}
   \ls \l_i^{-\frac16-s_p} \|v_i\|_{\XT},
   \quad i=4,5.
\]
  Finally, the bilinear estimate provides
\[ 
   \|v_{3,\l_3}u_\mu\|_{\L{15/7}{15/7}}
   \ls\l_3^{\frac{1}{30}-s_p}\mu^{-\frac{9}{10}}\|v_5\|_{\XT}\|u_\mu\|_{\Vkdv^2}.
\]
  The frequencies can estimated by
  $\l_2^{\frac{1}{15}}\l_5^{-\frac16-s_p}\mu^{-\frac{9}{10}}$.

  Note that we may change the role of $v_{3,\l_3}$ and $v_{4,\l_4}$ in the
  calculation above and hence can also estimate $v_{3,\l_3}$ in $\L{6}{6}$.  
\end{proof}

\section{Proof of the theorem}\label{sect:proof}
\noindent
In this section we are going to prove \thmref{thm:lwp}. The solution
$\psi=v+w$ of \eqref{eq:i_gkdv} is constructed by studying the
following equation
\[
   \left\{
     \begin{array}{rcl}
       \p_tw + \p_{xxx}w + \p_x\bigl(|v+w|^{p-1}(v+w)\bigr) & = & 0, \\
       w(0,x) & = & 0,
     \end{array}
   \right.
\]
where $v$ is a solution to the Airy equation \eqref{eq:i_airy}.

\begin{lemma}
\label{lem:contract}
  Let $W\in\XT$ and furthermore let $r$ and $\kappa_1$ be the constants from
  \lemref{lem:multlin_est} and \lemref{lem:airy_est_XT}, respectively. Under
  the assumptions of \thmref{thm:lwp}, we consider
\begin{equation}
\label{eq:eqn_contr}
   \left\{
     \begin{array}{rcl}
       \p_tw + \p_{xxx}w + \p_x\bigl(|v+W|^{p-1}(v+W)\bigr) & = & 0, \\
       w(0,x) & = & 0.
     \end{array}
   \right.
\end{equation}
  If $w$ solves \eqref{eq:eqn_contr} and $\|W\|_{\XT}\leq\alpha$, then there
  exists some $c>0$ such that for
\[
   \alpha\leq\min\left\{\kappa_1r_0,\frac{1}{2cr\kappa_1^pr_0^{p-1}}\right\}
   \quad\text{and}\quad
   \delta_0\leq\alpha,
\]
  it holds
\[
   \|w\|_{\XT}\leq\alpha.
\]
\end{lemma}
\begin{proof}
  For $\tau\in\R$ and $\l\in1.01^{\Z}$ we set $F_\l^\tau(u)=u_{<\l}+\tau u_\l$.
  Using that we define for $\ttau=(\tau_1,\ldots,\tau_n)\in\R^n$ and
  $\Ll=(\l_1,\ldots,\l_n)\in\bigl(1.01^{\Z}\bigr)^n$
\[
   \Ff_\Ll^\ttau(u)=F_{\l_1}^{\tau_1}\circ\cdots\circ F_{\l_n}^{\tau_n}(u).
\]
  One easily proves, that for $\ttau=(\tau_1,\ldots,\tau_n)\in[0,1]^n$, 
  $\Ll=(\l_1,\ldots,\l_n)\in(1.01^{\Z})^n$ and $\mu\in1.01^\Z$ we have
\[
   \bigl|\bigl(\Ff_\Ll^\ttau(u)\bigr)_\mu\bigr| \leq 2^n |u_\mu|
   \quad\text{and}\quad
   \bigl\|\Ff_{\Ll}^\ttau(u)\bigr\|_{\XTs} \leq 2^n\|u\|_{\XTs},\quad s\in\R.
\]
  Furthermore, one trivially verifies
\[
   \Ff_{\Ll}^\ttau(u+v) = \Ff_{\Ll}^\ttau(u)+\Ff_{\Ll}^\ttau(v).
\] 
  Set $f_p(x)=|x|^{p-1}x$, then by the telescoping series we have
\[
   f_p(u)=\sum_{\l\in1.01^{\Z}}\bigl( f_p(u_{\leq\l}) - f_p(u_{<\l}) \bigr).
\]
  By a standard trick, using the fundamental theorem of calculus we get
\begin{align*}
   f_p(u)
   &= \sum_{\l\in1.01^{\Z}}\int_0^1
     f_p'\bigl( u_{<\l} + \tau(u_{\leq\l} - u_{<\l}) \bigr)\d{\tau}
     (u_{\leq\l} - u_{<\l})\\
   &= \sum_{\l\in1.01^{\Z}}\int_0^1
     f_p'\bigl( F_\l^\tau(u) \bigr)\d{\tau}u_{\l}
\end{align*}
  In the sequel we use a more compact notation and write
\[
   \mathbf{u}_{\Ll}^{\ttau} := \Ff_\Ll^\ttau(u).
\]
  Reapplying this method three times, we get for $\Ll_i=(\l_i,\ldots,\l_5)$
  and $\ttau_i=(\tau_i,\ldots,\tau_5)$, $i=2,\ldots,5$, 
\[
  f_p(u) = \sum_{\substack{\l_2\leq\ldots\leq\l_5\\\l_i\in1.01^\Z}} \int_{[0,1]^4}
    f_p^{(4)}\bigl(\mathbf{u}_{\Ll_2}^{\ttau_2}\bigr)
    \mathbf{u}_{\Ll_3,\l_2}^{\ttau_3}
    \mathbf{u}_{\Ll_4,\l_3}^{\ttau_4}
    \mathbf{u}_{\Ll_5,\l_4}^{\ttau_5}
    u_{\l_5} \d{\ttau_2}.
\]
  That is in our context
\[
  |u|^{p-1}u = c\sum_{\substack{\l_2\leq\ldots\leq\l_5\\\l_i\in1.01^\Z}} \int_{[0,1]^4}
    \bigl|\mathbf{u}_{\Ll_2}^{\ttau_2}\bigr|^{p-5}
    \mathbf{u}_{\Ll_2}^{\ttau_2}
    \mathbf{u}_{\Ll_3,\l_2}^{\ttau_3}
    \mathbf{u}_{\Ll_4,\l_3}^{\ttau_4}
    \mathbf{u}_{\Ll_5,\l_4}^{\ttau_5}
    u_{\l_5} \d{\ttau_2},
\]
  where $c=p(p-1)(p-2)(p-3)(p-4)$.
  
  Let $w$ be a solution to \eqref{eq:eqn_contr}. By \lemref{lem:airy_est_XT}
  it suffices to show
\[
   \sup_{\mu\in1.01^\Z}\mu^{s_p}
   \left\| \int_{-\infty}^t e^{-(t-s)\p_x^3}
     \bigl( \p_x\bigl( |v+W|^{p-1}(v+W) \bigr) \bigr)_\mu(s)ds
   \right\|_{\Ukdv^2} \leq \frac{\alpha}{\kappa_1}.
\]
  By duality (cf.\ \corref{cor:duality_u_v}), it suffices to show that for
  each $\mu\in1.01^\Z$ we have
\[
   \frac{\mu^{s_p}}{\|u_\mu\|_{\Vkdv^2}}
   \left| \int \p_x\bigl( |v+W|^{p-1}(v+W) \bigr) u_\mu \dxdt
   \right| \leq \frac{\alpha}{\kappa_1}.
\]
  If we apply the calculation above once, we can rewrite the modulus of
  the integral as
\begin{multline*}
   \mathbf{S_1}+\mathbf{S_2}:=
   \sum_{\l_5:\;1.1\l_5>\mu}
     \left|
       \int \p_x\bigl(\bigl|\mathbf{v}_{\Ll_5}^{\ttau_5} 
         +\mathbf{W}_{\Ll_5}^{\ttau_5}\bigr|^{p-1}(v+W)_{\l_5}\bigr)u_\mu
         \dxdt\d{\ttau_5}
     \right|\\
   +\sum_{\l_5:\;1.1\l_5\leq\mu}
     \left|
       \int \p_x\bigl(\bigl|\mathbf{v}_{\Ll_5}^{\ttau_5} 
         +\mathbf{W}_{\Ll_5}^{\ttau_5}\bigr|^{p-1}(v+W)_{\l_5}\bigr)u_\mu
         \dxdt\d{\ttau_5}
     \right|.
\end{multline*}
  First, we consider the sum $\mathbf{S_1}$. We integrate by parts,
  apply the calculation above to
  $\bigl|\mathbf{v}_{\Ll_5}^{\ttau_5}+\mathbf{W}_{\Ll_5}^{\ttau_5}\bigr|^{p-1}$
  and hence have to bound
\begin{multline*}
  \frac{\mu^{1+s_p}}{\|u_\mu\|_{\Vkdv^2}}
  \sum_{\substack{\l_2\leq\ldots\leq\l_5\\\l_5:\;1.1\l_5>\mu}}
  \biggl|\int
    \bigl|\mathbf{v}_{\Ll_2}^{\ttau_2}
       +\mathbf{W}_{\Ll_2}^{\ttau_2}\bigr|^{p-5}
    \bigl(\mathbf{v}_{\Ll_2}^{\ttau_2} +\mathbf{W}_{\Ll_2}^{\ttau_2}\bigr) \\
    \times
    \bigl(\mathbf{v}_{\Ll_3}^{\ttau_3} +\mathbf{W}_{\Ll_3}^{\ttau_3}\bigr)_{\l_2}
    \cdots
    \bigl(\mathbf{v}_{\Ll_5}^{\ttau_5} +\mathbf{W}_{\Ll_5}^{\ttau_5}\bigr)_{\l_4}
    (v+W)_{\l_5}
    \tfrac{\p_x}{\mu}u_\mu\dxdt\d{\ttau_2}
   \biggr|.
\end{multline*}
  Note that since the spaces $\Vkdv^2$ are based on $L^2$, the operator
  $\frac{\p_x}{\mu}$ is bounded.
  We expand the factor
  $\bigl(\mathbf{v}_{\Ll_5}^{\ttau_5} +\mathbf{W}_{\Ll_5}^{\ttau_5}\bigr)_{\l_4}$.
  For $\mathbf{v}_{\Ll_5,\l_4}^{\ttau_5}$ 
  we apply \lemref{lem:multlin_est} and keep this factor in $\L{6}{6}$.
  For $\mathbf{W}_{\Ll_5,\l_4}^{\ttau_5}$
  we apply \lemref{lem:multlin_est} and estimate all terms in $\XT$.
  Hence, after summing over the frequencies, we obtain that $\mathbf{S_1}$
  is less than
\begin{multline*}
  \int_{[0,1]^4}
  r\bigl\|\mathbf{v}_{\Ll_2}^{\ttau_2} +\mathbf{W}_{\Ll_2}^{\ttau_2}\bigr\|_{\XT}^{p-4}
  \bigl\|\mathbf{v}_{\Ll_3}^{\ttau_3} +\mathbf{W}_{\Ll_3}^{\ttau_3}\bigr\|_{\XT}
  \bigl\|\mathbf{v}_{\Ll_4}^{\ttau_4} +\mathbf{W}_{\Ll_4}^{\ttau_4}\bigr\|_{\XT}
  \|v+W\|_{\XT}\\
  \times\left(
   \sup_{\l\in1.01^\Z}\l^{\frac16+s_p}
     \bigl\|\mathbf{v}_{\Ll_5,\l}^{\ttau_5}\bigr\|_{L^6([0,T],\R)}
   +\|\mathbf{W}_{\Ll_5}^{\ttau_5}\|_{\XT}
  \right)\d{\ttau_2}.
\end{multline*}
  By the properties of $\Ff_\Ll^\ttau(\cdot)$, we may estimate
  this by
\[
   cr\|v+W\|_{\XT}^{p-1}\biggl(
    \sup_{\l\in1.01^Z}\l^{\frac16+s_p}\|v_\l\|_{L^6([0,T],\R)}+\|W\|_{\XT}
   \biggr).
\]
  Using the bounds given in \thmref{thm:lwp} and $\alpha\leq\kappa_1r_0$ we
  may estimate this by
\[
   cr(r_0\kappa_1)^{p-1}(\delta_0+\alpha)
\]
  for some $c>0$.
  Since $\delta_0\leq\alpha$ and $\alpha\leq\frac{1}{2cr\kappa_1^pr_0^{p-1}}$
  we obtain
\[
   cr(r_0\kappa_1)^{p-1}(\delta_0+\alpha) \leq \frac{\alpha}{\kappa_1},
\]
  which implies the desired estimate $\|w\|_{\XT}\leq\alpha$ for
  $\mathbf{S_1}$.

  Now, we consider $\mathbf{S_2}$. Note that $\mathbf{S_2}=0$ if $p=5$,
  since the frequencies do not sum up to zero.
  Hence we may assume $p>5$ in the following.
  In order to estimate $\mathbf{S_2}$, we decompose
\begin{multline*}
  \bigl|\mathbf{v}_{\Ll_5}^{\ttau_5}
    +\mathbf{W}_{\Ll_5}^{\ttau_5}\bigr|^{p-1}(v+W)_{\l_5}\\
  =\sum_{\l_3\leq\l_4\leq\l_5}\int
    \bigl|\mathbf{v}_{\Ll_3}^{\ttau_3}+\mathbf{W}_{\Ll_3}^{\ttau_3}\bigr|^{p-3}
    \bigl(\mathbf{v}_{\Ll_4}^{\ttau_4}+\mathbf{W}_{\Ll_4}^{\ttau_4}\bigr)_{\l_3}
    \bigl(\mathbf{v}_{\Ll_5}^{\ttau_5}+\mathbf{W}_{\Ll_5}^{\ttau_5}\bigr)_{\l_4}
    (v+W)_{\l_5} \d{\ttau_3}.
\end{multline*}
  Differentiating this term with respect to $x$ yields
\begin{multline*}
  c\l_3\bigl|\mathbf{v}_{\Ll_3}^{\ttau_3}+\mathbf{W}_{\Ll_3}^{\ttau_3}\bigr|^{p-5}
  \bigl(\mathbf{v}_{\Ll_3}^{\ttau_3}+\mathbf{W}_{\Ll_3}^{\ttau_3}\bigr)
  \tfrac{\p_x}{\l_3}\bigl(\mathbf{v}_{\Ll_3}^{\ttau_3}+\mathbf{W}_{\Ll_3}^{\ttau_3}\bigr)
  \bigl(\mathbf{v}_{\Ll_4}^{\ttau_4}+\mathbf{W}_{\Ll_4}^{\ttau_4}\bigr)_{\l_3}
  \bigl(\mathbf{v}_{\Ll_5}^{\ttau_5}+\mathbf{W}_{\Ll_5}^{\ttau_5}\bigr)_{\l_4}
  (v+W)_{\l_5}\\
  \begin{aligned}
    &+\l_3\bigl|\mathbf{v}_{\Ll_3}^{\ttau_3}+\mathbf{W}_{\Ll_3}^{\ttau_3}\bigr|^{p-3}
    \tfrac{\p_x}{\l_3}
    \bigl(\mathbf{v}_{\Ll_4}^{\ttau_4}+\mathbf{W}_{\Ll_4}^{\ttau_4}\bigr)_{\l_3}
    \bigl(\mathbf{v}_{\Ll_5}^{\ttau_5}+\mathbf{W}_{\Ll_5}^{\ttau_5}\bigr)_{\l_4}
    (v+W)_{\l_5}\\
    &+\l_4\bigl|\mathbf{v}_{\Ll_3}^{\ttau_3}+\mathbf{W}_{\Ll_3}^{\ttau_3}\bigr|^{p-3}
    \bigl(\mathbf{v}_{\Ll_4}^{\ttau_4}+\mathbf{W}_{\Ll_4}^{\ttau_4}\bigr)_{\l_3}
    \tfrac{\p_x}{\l_3}
    \bigl(\mathbf{v}_{\Ll_5}^{\ttau_5}+\mathbf{W}_{\Ll_5}^{\ttau_5}\bigr)_{\l_4}
    (v+W)_{\l_5}\\
    &+\l_5\bigl|\mathbf{v}_{\Ll_3}^{\ttau_3}+\mathbf{W}_{\Ll_3}^{\ttau_3}\bigr|^{p-3}
    \bigl(\mathbf{v}_{\Ll_4}^{\ttau_4}+\mathbf{W}_{\Ll_4}^{\ttau_4}\bigr)_{\l_3}
    \bigl(\mathbf{v}_{\Ll_5}^{\ttau_5}+\mathbf{W}_{\Ll_5}^{\ttau_5}\bigr)_{\l_4}
    \tfrac{\p_x}{\l_5}(v+W)_{\l_5}.
  \end{aligned}
\end{multline*}
  We estimate all these terms exactly as for $\mathbf{S_1}$, but using
  \lemref{lem:multlin_est2} instead of \lemref{lem:multlin_est}. Note
  that the additional factor $\l_i$ on each term ensures that the
  summation over the frequencies converges. Note also that the operator
  $\frac{\p_x}{\l_i}$ does not play a role, since the $\Vkdv^2$ spaces
  are based on $L^2$. By the same argument as before, we can bound
  $\mathbf{S_2}$ by $\frac{\alpha}{\kappa_1}$ as before.
\end{proof}

\begin{proof}[Proof of \thmref{thm:lwp}]
  In order to prove \thmref{thm:lwp}, we use a fixed point argument
  to show existence and uniqueness. Let
\[
   \alpha\leq\min\left\{\kappa_1r_0,\frac{1}{2cr\kappa_1^pr_0^{p-1}}\right\},
   \quad
   \delta_0\leq\alpha
   \quad\text{and}\quad
   \|w_0\|_{\XT}<\alpha.
\]
  Furthermore, let
  \begin{gather*}
    \left\{
    \begin{array}{rcl}
    \p_t w_1+\p_{xxx}w_1+\p_x\bigl(|v+w_0|^{p-1}(v+w_0)\bigr) & = & 0, \\
    w_1(0,x) & = & 0,
    \end{array}
    \right.
    \label{eqn:wp_pf_iteration1_real} \\
	\intertext{and}
    \left\{
    \begin{array}{rcl}
    \p_t w_2+\p_{xxx}w_2+\p_x\bigl(|v+w_1|^{p-1}(v+w_1)\bigr) & = & 0, \\
    w_2(0,x) & = & 0,
    \end{array}
    \right.
    \label{eqn:wp_pf_iteration2_real}
  \end{gather*}
  be two iteration steps. Note that \lemref{lem:contract} ensures that
  $\|w_1\|_{\XT}<\alpha$ as well.
  We have to show that there exists $q\in(0,1)$ such that
\[
  \|w_2-w_1\|_{\XT} \leq q \|w_1-w_0\|_{\XT}.
\]
  By \lemref{lem:airy_est_XT} and duality, it suffices to replace the left hand side by
\[
   \kappa_1\frac{\mu^{s_p}}{\|u_\mu\|_{\Vkdv^2}}
   \biggl| \int \p_x\bigl(|v+w_0|^{p-1}(v+w_0)
   -|v+w_1|^{p-1}(v+w_1)\bigr)u_\mu\dxdt
   \biggr|.
\]
  For brevity we define $\omega_0=v+w_0$ and $\omega_1=v+w_1$.
  Similar as in the proof of \lemref{lem:contract} we may write
\[
  |\omega_0|^{p-1}\omega_0 - |\omega_1|^{p-1}\omega_1
  =\sum_{\l_5\in1.01^\Z}\int
   \bigl|\mathbf{\oo_0}_{\Ll_5}^{\ttau_5}\bigr|^{p-1}\omega_{0,\l_5}
  -\bigl|\mathbf{\oo_1}_{\Ll_5}^{\ttau_5}\bigr|^{p-1}\omega_{1,\l_5}\d{\ttau_5}.
\]
  Again, we split the sum into two parts, such that
\begin{align*}
  \mathbf{S_1}+\mathbf{S_2}=\sum_{\l_5:\;1.1\l_5>\mu}
  \left|\int\p_x\bigl(
  \bigl|\mathbf{\oo_0}_{\Ll_5}^{\ttau_5}\bigr|^{p-1}\omega_{0,\l_5}
  -\bigl|\mathbf{\oo_1}_{\Ll_5}^{\ttau_5}\bigr|^{p-1}\omega_{1,\l_5}
  \bigr)u_\mu\dxdt\d{\ttau_5}\right|\\
  +\sum_{\l_5:\;1.1\l_5\leq\mu}
  \left|\int\p_x\bigl(
  \bigl|\mathbf{\oo_0}_{\Ll_5}^{\ttau_5}\bigr|^{p-1}\omega_{0,\l_5}
  -\bigl|\mathbf{\oo_1}_{\Ll_5}^{\ttau_5}\bigr|^{p-1}\omega_{1,\l_5}
  \bigr)u_\mu\dxdt\d{\ttau_5}\right|.
\end{align*}
  First, we consider $\mathbf{S_1}$. Similar to the previous Lemma, we
  integrate by parts such that the derivative turns into a factor $\mu$,
  and we decompose
\[
  \int_{[0,1]}
  \bigl|\mathbf{\oo_0}_{\Ll_5}^{\ttau_5}\bigr|^{p-1}\omega_{0,\l_5}
  -\bigl|\mathbf{\oo_1}_{\Ll_5}^{\ttau_5}\bigr|^{p-1}\omega_{1,\l_5}
  \d{\ttau_5}
\]
  to
\begin{multline*}
   \sum_{\l_2\leq\ldots\leq\l_5}\int_{[0,1]^4}
   \bigl|\mathbf{\oo_0}_{\Ll_2}^{\ttau_2}\bigr|^{p-5}
   \mathbf{\oo_0}_{\Ll_2}^{\ttau_2}
   \mathbf{\oo_0}_{\Ll_3,\l_2}^{\ttau_3}
   \mathbf{\oo_0}_{\Ll_4,\l_3}^{\ttau_4}
   \mathbf{\oo_0}_{\Ll_5,\l_4}^{\ttau_5}
   \omega_{0,\l_5}
    \\
   -\bigl|\mathbf{\oo_1}_{\Ll_2,}^{\ttau_2}\bigr|^{p-5}
   \mathbf{\oo_1}_{\Ll_2}^{\ttau_2}
   \mathbf{\oo_1}_{\Ll_3,\l_2}^{\ttau_3}
   \mathbf{\oo_1}_{\Ll_4,\l_3}^{\ttau_4}
   \mathbf{\oo_1}_{\Ll_5,\l_4}^{\ttau_5}
   \omega_{1,\l_5}
   \d{\ttau_2}.
\end{multline*}
  The integrand may be written as
\begin{multline*}
   \bigl|\mathbf{\oo_0}_{\Ll_2}^{\ttau_2}\bigr|^{p-5}
   \mathbf{\oo_0}_{\Ll_2}^{\ttau_2}
   \mathbf{\oo_0}_{\Ll_3,\l_2}^{\ttau_3}
   \mathbf{\oo_0}_{\Ll_4,\l_3}^{\ttau_4}
   \mathbf{\oo_0}_{\Ll_5,\l_4}^{\ttau_5}
   (\omega_{0,\l_5}-\omega_{1,\l_5}) \\
  \begin{aligned}
   &\quad+\bigl|\mathbf{\oo_0}_{\Ll_2}^{\ttau_2}\bigr|^{p-5}
   \mathbf{\oo_0}_{\Ll_2}^{\ttau_2}
   \mathbf{\oo_0}_{\Ll_3,\l_2}^{\ttau_3}
   \mathbf{\oo_0}_{\Ll_4,\l_3}^{\ttau_4}
   \bigl(\mathbf{\oo_0}_{\Ll_5,\l_4}^{\ttau_5}
     -\mathbf{\oo_1}_{\Ll_5,\l_4}^{\ttau_5}\bigr)
   \omega_{1,\l_5} \\
   &\quad+\ldots \\
   &\quad+\bigl(\bigl|\mathbf{\oo_0}_{\Ll_2}^{\ttau_2}\bigr|^{p-5}
     \mathbf{\oo_0}_{\Ll_2}^{\ttau_2}
    -\bigl|\mathbf{\oo_1}_{\Ll_2}^{\ttau_2}\bigr|^{p-5}
     \mathbf{\oo_1}_{\Ll_2}^{\ttau_2}\bigr)
   \mathbf{\oo_1}_{\Ll_3,\l_2}^{\ttau_3}
   \mathbf{\oo_1}_{\Ll_4,\l_3}^{\ttau_4}
   \mathbf{\oo_1}_{\Ll_5,\l_4}^{\ttau_5}
   \omega_{1,\l_5} \\
   &=: S_1 + \ldots + S_5
  \end{aligned}
\end{multline*}
  Using the fundamental theorem of calculus, we can further manipulate the
  last term $S_5$ to get
\begin{multline*}
   S_5=c\int_0^1 \bigl|
     \mathbf{\oo_0}_{\Ll_2}^{\ttau_2} + \tau
     \bigl(\mathbf{\oo_1}_{\Ll_2}^{\ttau_2}
      -\mathbf{\oo_0}_{\Ll_2}^{\ttau_2}\bigr)
   \bigr|^{p-5}\d{\tau}
   \bigl(\mathbf{\oo_0}_{\Ll_2}^{\ttau_2}
     -\mathbf{\oo_1}_{\Ll_2}^{\ttau_2}\bigr)\\
   \times\mathbf{\oo_1}_{\Ll_3,\l_2}^{\ttau_3}
   \mathbf{\oo_1}_{\Ll_4,\l_3}^{\ttau_4}
   \mathbf{\oo_1}_{\Ll_5,\l_4}^{\ttau_5}
   \omega_{1,\l_5}.
\end{multline*}
  We split $S_1$ into two terms by expanding
  $\mathbf{\oo_0}_{\Ll_5,\l_4}^{\ttau_5}$:
\begin{multline*}
   S_1=\bigl|\mathbf{\oo_0}_{\Ll_2}^{\ttau_2}\bigr|^{p-5}
   \mathbf{\oo_0}_{\Ll_2}^{\ttau_2}
   \mathbf{\oo_0}_{\Ll_3,\l_2}^{\ttau_3}
   \mathbf{\oo_0}_{\Ll_4,\l_3}^{\ttau_4}
   \mathbf{v}_{\Ll_5,\l_4}^{\ttau_5}
   (w_{0,\l_5}-w_{1,\l_5}) \\
   +\bigl|\mathbf{\oo_0}_{\Ll_2}^{\ttau_2}\bigr|^{p-5}
   \mathbf{\oo_0}_{\Ll_2}^{\ttau_2}
   \mathbf{\oo_0}_{\Ll_3,\l_2}^{\ttau_3}
   \mathbf{\oo_0}_{\Ll_4,\l_3}^{\ttau_4}
   \mathbf{w_0}_{\Ll_5,\l_4}^{\ttau_5}
   (w_{0,\l_5}-w_{1,\l_5}).
\end{multline*}
  For the first term we estimate
  $\mathbf{v}_{\Ll_5,\l_4}^{\ttau_5}$ in $\L{6}{6}$, and for the second
  term we estimate all factors in $\XT$. Hence, by \lemref{lem:multlin_est}
\[
   \frac{\mu^{1+s_p}}{\|u_\mu\|_{\Vkdv^2}}
    \sum_{\substack{\l_2\leq\ldots\leq\l_5\\\l_5:\;1.1\l_5\leq\mu}}
     \left| \int S_1\tfrac{\p_x}{\mu}u_\mu \dxdt\d{\ttau_2} \right|
     \ls
     r(r_0\kappa_1)^{p-2}(\delta_0+\alpha)\|w_0-w_1\|_{\XT}.
\]
  Analogously, we expand either $\mathbf{\oo_0}_{\Ll_4,\l_3}^{\ttau_4}$
  or $\mathbf{\oo_1}_{\Ll_5,\l_4}^{\ttau_5}$ in $S_2,\ldots,S_5$.
  For each $S_i$, $i=2,\ldots,5$, either the expanded term depends on $v$,
  then we choose to estimate this factor in $\L{6}{6}$, or we estimate all
  factors in $\XT$. Thus, by \lemref{lem:multlin_est} we estimate for
  $i=2,\ldots,5$
\[
  \frac{\mu^{1+s_p}}{\|u_\mu\|_{\Vkdv^2}}
    \sum_{\substack{\l_2\leq\ldots\leq\l_5\\\l_5:\;1.1\l_5\leq\mu}}
    \left| \int S_i\tfrac{\p_x}{\mu}u_\mu \dxdt\d{\ttau_2} \right|
   \ls
   r(r_0\kappa_1)^{p-2}(\delta_0+\alpha)\|w_0-w_1\|_{\XT}.
\]
  All in all, we obtain
\[
  \kappa_1
    \frac{\mu^{s_p}}{\|u_\mu\|_{\Vkdv^2}} \mathbf{S_1}
  \leq cr\kappa_1(r_0\kappa_1)^{p-2}(\delta_0+\alpha)\|w_0-w_1\|_{\XT}.
\]
  Now, we may choose $\alpha$ (and hence $\delta_0$) small enough such that
\[
   cr\kappa_1(r_0\kappa_1)^{p-2}(\delta_0+\alpha) < \frac12,
\]
 which gives the desired estimate for $\mathbf{S_1}$.

 Now, we consider
\[
  \mathbf{S_2}=\sum_{\l_5:\;1.1\l_5\leq\mu}
  \left|\int\p_x\bigl(
  \bigl|\mathbf{\oo_0}_{\Ll_5}^{\ttau_5}\bigr|^{p-1}\omega_{0,\l_5}
  -\bigl|\mathbf{\oo_1}_{\Ll_5}^{\ttau_5}\bigr|^{p-1}\omega_{1,\l_5}
  \bigr)u_\mu\dxdt\d{\ttau_5}\right|.
\]
  Note that if $p=5$ then $\mathbf{S_2}=0$ by the same argument as in
  \lemref{lem:contract}. Hence, we may assume $p>5$.
  We decompose 
\begin{multline*}
  \int_{[0,1]}\bigl|\mathbf{\oo_0}_{\Ll_5}^{\ttau_5}\bigr|^{p-1}\omega_{0,\l_5}
  -\bigl|\mathbf{\oo_1}_{\Ll_5}^{\ttau_5}\bigr|^{p-1}\omega_{1,\l_5} \d{\ttau_5}\\
  \begin{aligned}
  &=\sum_{\l_3\leq\l_4\leq\l_5}\int
    \bigl|\mathbf{\oo_0}_{\Ll_3}^{\ttau_3}\bigr|^{p-3}
    \mathbf{\oo_0}_{\Ll_4,\l_3}^{\ttau_4} \mathbf{\oo_0}_{\Ll_5,\l_4}^{\ttau_5}
    \omega_{0,\l_5}
   -\bigl|\mathbf{\oo_1}_{\Ll_3}^{\ttau_3}\bigr|^{p-3}
    \mathbf{\oo_1}_{\Ll_4,\l_3}^{\ttau_4} \mathbf{\oo_1}_{\Ll_5,\l_4}^{\ttau_5}
    \omega_{1,\l_5} \d{\ttau_3}\\
  &=\sum_{\l_3\leq\l_4\leq\l_5}\int
    \bigl|\mathbf{\oo_0}_{\Ll_3}^{\ttau_3}\bigr|^{p-3}
    \mathbf{\oo_0}_{\Ll_4,\l_3}^{\ttau_4} \mathbf{\oo_0}_{\Ll_5,\l_4}^{\ttau_5}
    (\omega_{0,\l_5}-\omega_{1,\l_5})\\
   &\qquad+\ldots\\
   &\qquad+\bigl(\bigl|\mathbf{\oo_0}_{\Ll_3}^{\ttau_3}\bigr|^{p-3}
           -\bigl|\mathbf{\oo_1}_{\Ll_3}^{\ttau_3}\bigr|^{p-3}\bigr)
    \mathbf{\oo_1}_{\Ll_4,\l_3}^{\ttau_4} \mathbf{\oo_1}_{\Ll_5,\l_4}^{\ttau_5}
    \omega_{1,\l_5} \d{\ttau_3}\\
   &=:\sum_{\l_3\leq\l_4\leq\l_5}\int S_1+\ldots+S_4 \d{\ttau_3}
  \end{aligned}
\end{multline*}
  If we differentiate $S_1$ with respect to $x$, then we are able
  to apply \lemref{lem:multlin_est2} and since we obtain an additional factor
  $\l_i$, we get
\begin{multline*}
  \kappa_1\frac{\mu^{s_p}}{\|u_\mu\|_{\Vkdv^2}}
  \sum_{\substack{\l_3\leq\l_4\leq\l_4\\\l_5:\;1.1\l_5\leq\mu}}
  \left|\int S_1u_\mu \dxdt \d{\ttau_3} \right|\\
  \begin{aligned}
   &\leq cr\bigl\|\mathbf{\oo_0}_{\Ll_3}^{\ttau_3}\bigr\|_{\XT}^{p-3}
   \bigl\|\mathbf{\oo_0}_{\Ll_4}^{\ttau_4}\bigr\|_{\XT}
   \biggl(\sup_{\l\in1.01^\Z}\l^{\frac16+s_p}\bigl\|
      \mathbf{v}_{\Ll_5,\l}^{\ttau_5}\bigr\|_{L([0,T],\R)}
    +\bigl\|\mathbf{w_0}_{\Ll_5}^{\ttau_5}\bigr\|_{\XT}\biggr)\\
   &\qquad\times\|\omega_{0}-\omega_{1}\|_{\XT}\\
   &\leq cr\kappa_1(r_0\kappa_1)^{p-2}(\delta_0+\alpha)
    \|w_0-w_1\|_{\XT}.
  \end{aligned}
\end{multline*}
  We can treat $S_2$ and $S_3$ analogously. Now, we consider $S_4$.
  Applying the fundamental theorem of calculus, we obtain
  for $\OO(\tau)=\mathbf{\oo_0}_{\Ll_3}^{\ttau_3}
       +\tau\bigl(\mathbf{\oo_1}_{\Ll_3}^{\ttau_3}
        -\mathbf{\oo_0}_{\Ll_3}^{\ttau_3}\bigr)$:
\[
  S_4= c\int_0^1 |\OO(\tau)|^{p-5}\OO(\tau) \d{\tau}
     \bigl(\mathbf{\oo_0}_{\Ll_3}^{\ttau_3}-\mathbf{\oo_1}_{\Ll_3}^{\ttau_3}\bigr)
     \mathbf{\oo_1}_{\Ll_4,\l_3}^{\ttau_4} \mathbf{\oo_1}_{\Ll_5,\l_4}^{\ttau_5}
     \omega_{1,\l_5}.
\]
  Differentiating this term with respect to $x$ yields a sum of $5$ terms,
  each of which can be estimated using \lemref{lem:multlin_est2} as above.
  All in all we obtain
\[
  \kappa_1
    \frac{\mu^{s_p}}{\|u_\mu\|_{\Vkdv^2}} \mathbf{S_2}
  \leq cr\kappa_1(r_0\kappa_1)^{p-2}(\delta_0+\alpha)
  \|w_0-w_1\|_{\XT}.
\]
  By possibly chooser $\alpha$ smaller again, we have
\[
   cr\kappa_1(r_0\kappa_1)^{p-2}(\delta_0+\alpha) < \frac12.
\]

  Thus, for small enough $\alpha$, we have a contraction and Banach
  fixed-point theorem gives existence and uniqueness.
\end{proof}

The following proof of \corref{cor:gwp} is an observation of Koch and Marzuola in
\cite[p. 175-176]{KM12}.
\begin{proof}[Proof of \corref{cor:gwp}]
  By Strichartz' estimates for linear KdV and \lemref{lem:airy_est_XT}
  we have for $v$ given as in \eqref{eq:i_airy} and $0 < T \leq \infty$ that
\[
   \sup_{\l\in1.01^\Z} \l^{\frac16+s_p} \|v_\l\|_{L^6([0,T],L^6(\R))} 
   \leq \kappa_0\|v\|_{\XT}
   \leq \kappa_0\kappa_1\|\psi_0\|_{\B}
   \leq \delta_0(1).
\]
  Since this estimate holds true for all $0 < T \leq \infty$, we may apply
  \thmref{thm:lwp} with $T=\infty$ to obtain global existence.
\end{proof}

\nocite{BL76}
\providecommand{\bysame}{\leavevmode\hbox to3em{\hrulefill}\thinspace}
\providecommand{\MR}{\relax\ifhmode\unskip\space\fi MR }
\providecommand{\MRhref}[2]{%
  \href{http://www.ams.org/mathscinet-getitem?mr=#1}{#2}
}
\providecommand{\href}[2]{#2}

\bibliographystyle{amsplain}\label{sect:refs}

\end{document}